\documentclass{amsart}

\usepackage{amsmath}
\usepackage{amssymb}
\usepackage{amsthm}
\usepackage{mathrsfs}

\DeclareMathOperator{\Spec}{Spec}
\DeclareMathOperator{\Fr}{Fr}
\DeclareMathOperator{\Gal}{Gal}
\DeclareMathOperator{\id}{id}
\DeclareMathOperator{\tr}{tr}
\DeclareMathOperator{\proj}{proj}

\newtheorem{thm}{Theorem}[section]

\newtheorem{lem}[thm]{Lemma}

\newtheorem*{thmn}{Theorem}

\newcommand{\F}{\mathbb{F}}

\begin{document}

\title
[On the connected components of moduli spaces]
{On the connected components of moduli spaces 
of finite flat models}
\author{Naoki Imai}
\address{Graduate School of Mathematical Sciences, 
the University of Tokyo, 3-8-1 Komaba, Meguro-ku, Tokyo 153-8914, Japan.}
\email{naoki@ms.u-tokyo.ac.jp}

\maketitle

\begin{abstract}
We prove that the non-ordinary component is connected 
in the moduli spaces of finite flat models 
of two-dimensional local Galois representations over finite fields. 
This was conjectured by Kisin.
As an application to global Galois representations, 
we prove a theorem on the modularity 
comparing a deformation ring and a Hecke ring.
\end{abstract}

\section*{Introduction}
Let $K$ be a $p$-adic field for $p>2$, and let $V_{\F}$ be 
a two-dimensional continuous 
representation of the absolute Galois group $G_K$ 
over a finite field $\F$ of characteristic $p$. 
The projective
scheme $\mathscr{GR}^{\mathbf{v}}_{V_{\F},0}$ over $\F$
is the moduli of finite flat models of $V_{\F}$ with some
determinant condition. From the viewpoint 
of the application to the modularity problem, we are interested
in the connected components of
$\mathscr{GR}^{\mathbf{v}}_{V_{\F},0}$. The ordinary component
of $\mathscr{GR}^{\mathbf{v}}_{V_{\F},0}$ was determined in
\cite{Kis}, and Kisin conjectured that 
the non-ordinary component is connected. 
We prove this conjecture, and the main theorem is the following.

\begin{thmn}
Let $\F'$ be a finite extension of $\F$. 
Suppose $x_1 ,x_2 \in\mathscr{GR}^{\mathbf{v}}_{V_{\F},0}(\F')$ 
correspond to objects 
$\mathfrak{M}_{1,\F'} , \mathfrak{M}_{2,\F'}$ of 
$({\mathrm{Mod}}/\mathfrak{S})_{\F'}$ respectively.
If $\mathfrak{M}_{1,\F'}$ and $\mathfrak{M}_{2,\F'}$
 are both non-ordinary, then $x_1$ and $x_2$ lie on the same connected
 component of
 $\mathscr{GR}^{\mathbf{v}}_{V_{\F},0}$.
\end{thmn}

When $K$ is totally ramified over $\mathbb{Q}_p$, this was proved in
\cite{Kis}.
If the residue field of $K$ is bigger than $\F_p$, 
the situation changes greatly 
because $\mathfrak{S}\otimes _{\mathbb{Z}_p}{\F}$ can
be split into a direct product.
When $K$ is a general $p$-adic field, 
the case of $V_{\mathbb{F}}$ being the trivial
representation was treated in \cite{Gee}. 

As an application to global Galois representations, 
we prove a theorem on the modularity, 
which states that 
a deformation ring is isomorphic to a Hecke ring 
up to $p$-power torsion kernel. 
This completes Kisin's theory for $GL _2 $.

\subsection*{Acknowledgment}
The author is grateful to his advisor Takeshi Saito 
for his careful reading of an earlier version of this paper 
and for his helpful comments. 
He would like to thank the referee 
for a careful reading of this paper 
and a number of suggestions for improvements. 

\subsection*{Notation}
Throughout this paper, we use the following notation.
Let $p>2$ be a prime number, and $k$ be a finite extension of
$\mathbb{F}_p$ of cardinality $q=p^n$. 
The Witt ring of $k$ is denoted by $W(k)$, and let $K_0 =W(k)[1/p]$. 
Let $K$ be
a totally ramified extension of $K_0$ of degree $e$, and $\mathcal{O}_K$
be the ring of integers of $K$. Let $I_K$ be the inertia group of
the absolute Galois group $G_K$, and $\Fr _q$ be the $q$-th power Frobenius
of the absolute Galois group $G_k$. 
Let $\F$ be a finite field of characteristic $p$.
The formal power series ring of $u$ over $\F$ is
denoted by $\F [[u]]$, 
and its quotient field is denoted by $\F ((u))$. 
Let $v_u$ be the valuation of 
$\F ((u))$ normalized by $v_u (u)=1$.

\section{Preliminaries}
First of all, we recall some notation from \cite{Kis}, 
and the interested reader should consult \cite{Kis} 
for more detailed definitions.

For each $\mathbb{Q}_p$-algebra embedding $\psi : K \to \overline{K}_0$, 
we put $v_{\psi} =1$ and set $\mathbf{v} = (v_{\psi} )_{\psi}$.
Let $F$ be a finite Galois extension of $\mathbb{Q}_p$
containing $K_0$, and $\F$ be the residue field of $F$.
Let $V_{\F}$ be a continuous 
two-dimensional representation of $G_K$ over $\F$. We assume that $V_{\F}$
arises as the generic fiber of a finite flat group scheme over $\mathcal{O}_K$.

Let $R$ be a complete local Noetherian $\mathcal{O}_F$-algebra with
residue field $\F$, and fix $\xi \in D^{\mathrm{fl}}_{V_{\F}}(R)$. We make
the following assumption:
\begin{center}
The morphism $\xi \to D^{\mathrm{fl}}_{V_{\F}}$ of groupoids over
 $\mathfrak{AR}_{\mathcal{O}_F}$ is formally smooth.
\end{center}
Now we can construct 
$\varTheta^{\mathbf{v}}_{V_{\F},\xi} : 
 \mathscr{GR}^{\mathbf{v},{\mathrm{loc}}}_{V_{\F},\xi} \to 
 \Spec R^{\mathbf{v}}$. 
Let 
$\mathscr{GR}^{\mathbf{v}}_{V_{\F},0}$ 
be its fiber over the closed point of $\Spec R^{\mathbf{v}}$.
We assume $\Spec R^{\mathbf{v}} \neq \emptyset$, and this assumption
assures that the action of $I_K$ on $\det V_{\F}$ is
the reduction mod $p$ of the cyclotomic character.

The fundamental character of level $m$ is given by
\[
 \omega _m : I_K \to \overline{k}^{\times} ;\ 
 g\mapsto \frac{g(\sqrt[p^m -1]{\pi})}{\sqrt[p^m -1]{\pi}} 
 \ \mathrm{mod} \ m_{\mathcal{O}_{\overline{K}}}.
\]
Here $\pi$ is a uniformizer of $\mathcal{O}_K$, 
and $m_{\mathcal{O}_{\overline{K}}}$ is the maximal ideal of 
$\mathcal{O}_{\overline{K}}$. 
If $K'/K$ is a finite unramified extension that contains the
$(p^m -1)$-st roots of unity, then the same formula
as above defines a character of $G_{K'}$, which is again 
denoted by $\omega _m$. Note that this extension depends on the 
choice of the uniformizer $\pi$. From now on, we fix 
a uniformizer $\pi$.

\begin{lem}\label{absirr}
If $V_{\F}$ is absolutely irreducible and $\F_{q^2} \subset \F$, then
\[
 V_{\F}|_{I_K} \sim \omega _{2n} ^{s} \oplus \omega _{2n} ^{qs}
\]
for a positive integer $s$ such that $(q+1) \nmid s$.
\end{lem}

\begin{proof}
Let $I_P \subset I_K$ be the wild inertia group. 
Then $V^{I_P} _{\F} \neq 0$ and $V^{I_P} _{\F}$ is 
$G_K$-stable, so $V^{I_P} _{\F} = V_{\F}$.
As the action of $I_K$ on $V_{\F}$ factors through 
the tame inertia group, we get
$V_{\F}|_{I_K} \sim \omega _{m_1} ^{s_1}  \oplus \omega _{m_2} ^{s_2}$
for some non-negative integers $s_1$, $s_2$ 
and some positive integers $m_1$, $m_2$.
Now we fix a lifting 
$\widetilde{\Fr} _q \in G_K$ of the $q$-th Frobenius $\Fr _q$. 
For every $\sigma \in I_K$ and every positive integer $m$, we have 
$\omega _m \bigl( \widetilde{\Fr} _q \circ \sigma \circ
 (\widetilde{\Fr} _q) ^{-1} \bigr) =\omega _m (\sigma)^q$. 
Changing the above basis by the action of $(\widetilde{\Fr} _q )^{-1}$, 
we obtain 
$V_{\F}|_{I_K} \sim \omega _{m_1} ^{qs_1} \oplus \omega _{m_2} ^{qs_2}$.

If $\omega _{m_1} ^{s_1} = \omega _{m_2} ^{s_2} $, we get
$\omega _{m_1} ^{s_1} = \omega _{m_1} ^{qs_1} $. So 
we may assume $m_1=n$. As $\omega_n$ is defined over $G_K$, 
we can consider the representation $V_{\F} \otimes \omega _n ^{-s_1}$ 
of $G_K$. Then this representation is absolutely irreducible 
and factors through $G_k$. This is a contradiction.

So we may assume $\omega _{m_1} ^{s_1} \neq \omega _{m_2} ^{s_2}$.
As $V_{\F}$ is an irreducible representation, 
$\omega _{m_1} ^{s_1} = \omega _{m_2} ^{qs_2} $ and 
$\omega _{m_2} ^{s_2} = \omega _{m_1} ^{qs_1} $.
Hence $\omega _{m_1} ^{s_1} = \omega _{m_1} ^{q^2 s_1}$ 
and we may assume $m_1=2n$. Thus we get 
$V_{\F}|_{I_K} \sim \omega _{2n} ^{s} \oplus \omega _{2n} ^{qs}$.

If $(q+1) \mid s $, then 
$V_{\F}|_{I_K} \sim \omega ^{s'} _n \oplus \omega _n ^{s'} $ 
where $s'=s/(q+1)$.
This contradicts the absolutely irreducibility of $V_{\F}$ 
by considering $V_{\F} \otimes \omega _n ^{-s'}$. 
So we get $(q+1) \nmid s$.
\end{proof}

Let $\mathfrak{S} = W(k)[[u]] $, and $\mathcal{O}_{\mathcal{E}}$ be 
the $p$-adic completion of $\mathfrak{S} [1/u]$. 
We choose elements $\pi _m \in \overline{K}$ such that
$\pi _0 = \pi$ and $\pi ^p _{m+1} = \pi _m$ for $m \geq 0$, 
and put $K_{\infty} = \bigcup _{m\geq 0} K(\pi _m )$.
Let 
$M_{\F} \in \Phi {\mathrm{M}}_{\mathcal{O}_{\mathcal{E}},\F}$ 
be the $\phi$-module 
that corresponds to the $G_{K_{\infty}}$-representation 
$V_{\F} (-1)$. 
Here $(-1)$ denotes the inverse of the Tate twist. 

 From now on, we assume $\F_{q^2} \subset \F$ and 
fix an embedding $k \hookrightarrow \F$. 
This assumption does not matter, 
because we may extend $\F$ to prove the main theorem. 
We consider the isomorphism
\[
 \mathcal{O}_{\mathcal{E}} \otimes _{\mathbb{Z}_p} \F \cong 
 k((u))\otimes _{\F _p} \F \stackrel{\sim}{\to}
 \prod _{\sigma \in \Gal (k/{\F_p})} \F ((u))\ ;\ 
 \Bigl( \sum a_i u^i \Bigr) \otimes b \mapsto
 \Bigl( \sum \sigma (a_i ) b u^i \Bigr) _{\sigma}
\]
and let $\epsilon_{\sigma} \in k((u)) \otimes _{\F_p} \F $ be 
the primitive idempotent corresponding to $\sigma$.
Take $\sigma _1 , \cdots , \sigma _n \in \Gal (k/\F _p)$ such 
that $\sigma _{i+1} = \sigma _i \circ \phi ^{-1}$. 
Here we regard $\phi$ as the $p$-th power Frobenius, 
and use the convention that 
$\sigma _{n+i} =\sigma _i$.
In the following, we often use such conventions.
Then we have 
$\phi (\epsilon _{\sigma _i}) = \epsilon _{\sigma _{i+1}}$, 
and $\phi : M_{\F} \to M_{\F}$ determines 
$\phi :\epsilon _{\sigma _i} M_{\F} \to 
 \epsilon _{\sigma _{i+1}} M_{\F}$.
For $(A_i )_{1 \leq i \leq n} \in GL_2 \bigl( \F ((u)) \bigr)^n$, 
we write
\[
 M_{\F} \sim (A_1 ,A_2 , \ldots ,A_n )=(A_i )_i
\]
if there is a basis $\{ e^i _1 ,e^i _2 \}$ 
of $\epsilon _{\sigma _i} M_{\F}$ over $\F ((u))$ 
such that 
$\phi 
\begin{pmatrix} e^i _1 \\ e^i _2 \end{pmatrix}
 = A_i 
\begin{pmatrix} e^{i+1} _1 \\ e^{i+1} _2 \end{pmatrix}$.
We use the same notation for any sublattice 
$\mathfrak{M} _{\F} \subset M_{\F}$ similarly. 
Here and in the following, we consider only sublattices 
that are $\mathfrak{S}\otimes _{\mathbb{Z}_p} \F$-modules.

Finally, 
for any sublattice 
$\mathfrak{M} _{\F} \subset M_{\F}$ with a chosen 
basis $\{ e^i _1 ,e^i _2 \}_{1 \leq i \leq n}$ and 
$B=(B_i )_{1 \leq i \leq n} \in GL _2 \bigl(\F((u))\bigr)^n$, 
the module 
generated by the entries of 
$\biggl{\langle} B_i 
\begin{pmatrix} e^i _1 \\ e^i _2 \end{pmatrix} 
\biggr{\rangle}$ 
with the basis given by these entries 
is denoted by 
$B\cdot \mathfrak{M} _{\F}$. 
Note that $B\cdot \mathfrak{M} _{\F}$ depends on 
the choice of the basis of $\mathfrak{M} _{\F}$. 

\begin{lem}\label{stracture}
Suppose $V_{\F}$ is absolutely irreducible. 
If $\F'$ is the quadratic extension of $\F$, 
then 
\[
 M_{\F} \otimes _{\F} \F ' \sim 
\Biggl(
\begin{pmatrix}
 0 & \alpha _1 \\ \alpha _1 u^s & 0
\end{pmatrix}
,
\begin{pmatrix}
 \alpha _2 & 0 \\ 0 & \alpha _2
\end{pmatrix}
,
\ldots
,
\begin{pmatrix}
 \alpha _n & 0 \\ 0 & \alpha _n
\end{pmatrix}
\Biggr)
\]
for some $\alpha _i \in (\F')^{\times}$ and a positive integer $s$ 
such that $(q+1) \nmid s$. 
\end{lem}

\begin{proof}
Let $K'$ be the quadratic unramified extension of $K$, 
and $k'$ be the residue field of 
$K'$. Then 
\[
 V_{\F} (-1)|_{G_{K'}} \sim \lambda ' \omega _{2n} ^{-s} 
 \oplus \lambda ' \omega _{2n} ^{-qs}
\]
for an unramified character 
$\lambda ' :G_{K'} \to \F^{\times}$ 
and a positive integer $s$ such that $(q+1) \nmid s$ 
by applying Lemma \ref{absirr} to 
$V_{\F} (-1)^*$. 
By taking the quadratic extension $\F'$ of $\F$, 
we can extend $\lambda '$ 
to $\lambda :G_K \to (\F')^{\times}$.
We take a lifting $\widetilde{\Fr} _q \in G_{K_{\infty}}$ of 
the $q$-th Frobenius $\Fr _q$. 
Now we fix a $(q^2 -1)$-st root of $\pi$, 
which is denoted by $\sqrt[q^2 -1]{\pi}$. 
Then we put 
$\widetilde{\alpha} = 
 \widetilde{\Fr} _q (\sqrt[q^2 -1]{\pi})/
 \sqrt[q^2 -1]{\pi} \in \mathcal{O}_{\overline{K}}$,
and let $\alpha$ be the reduction of $\widetilde{\alpha}$ 
in $\overline{k}$. 
We have $\alpha \in \F '$, 
because $\alpha ^{q^2 -1} =1$.
Considering $V_{\F} \otimes _{\F} {\F'}$, we may assume 
$\F =\F '$.

We put $K' _{\infty} = K' \cdot K_{\infty}$. 
Then $(\widetilde{\Fr} _q )^2$ is in $G_{K' _{\infty}}$.
Now we have 
\[
 \frac{(\widetilde{\Fr} _q ) ^2 (\sqrt[q^2 -1]{\pi})}
 {\sqrt[q^2 -1]{\pi}} =
 \frac{(\widetilde{\Fr} _q ) ^2 (\sqrt[q^2 -1]{\pi})}
 {\widetilde{\Fr} _q (\sqrt[q^2 -1]{\pi})}
 \cdot
 \frac{\widetilde{\Fr} _q (\sqrt[q^2 -1]{\pi})}
 {\sqrt[q^2 -1]{\pi}}
 =\frac{\widetilde{\Fr} _q (\widetilde{\alpha} \sqrt[q^2 -1]{\pi}) }
 {\widetilde{\Fr} _q (\sqrt[q^2 -1]{\pi})}
 \widetilde{\alpha}
 =\widetilde{\Fr} _q (\widetilde{\alpha}) \widetilde{\alpha} 
\]  
and 
$\omega _{2n} \bigl( (\widetilde{\Fr} _q )^2 \bigr) = \alpha ^{q+1}$.
Hence we can take $v_1 ,v_2 \in V_{\F}(-1)$ 
so that 
\[
 \widetilde{\Fr} _q (v_1)=\lambda (\widetilde{\Fr} _q )
 \alpha ^{-qs} v_2,\ 
 \widetilde{\Fr} _q (v_2)=\lambda (\widetilde{\Fr} _q )
 \alpha ^{-s} v_1 
\] 
and 
\[
 g(v_1)=\lambda (g)\omega _{2n} ^{-s} (g)v_1,\ 
 g(v_2)=\lambda (g)\omega _{2n} ^{-qs} (g)v_2
\]
for all $g \in G_{K'_{\infty}}$.
We take an element 
$w_{\lambda}$ of $(\overline{k} \otimes _{\F _p} \F )^{\times}$ 
so that 
$g(w_{\lambda}) = 
\bigl( 1\otimes \lambda (g)\bigr) w_{\lambda}$ 
for all $g \in G_K$. 
By this condition, $w_{\lambda}$ is determined up to 
$(k\otimes _{\F _p} \F)^{\times}$.

By the definition of the action of $G_{K_{\infty}}$ on 
$\mathcal{O}_{\mathcal{E}^{\mathrm{ur}}}$, 
we can choose an element $u_{2n}$ 
of $\mathcal{O}_{\mathcal{E}^{\mathrm{ur}}}
 /p\mathcal{O}_{\mathcal{E}^{\mathrm{ur}}}$
so that $u_{2n} ^{q^2 -1} = u$ and 
$\widetilde{\Fr} _q (u_{2n}) = \alpha u_{2n}$.
We consider the isomorphism 
\[
 k' \otimes _{\F_p} \F \stackrel{\sim}{\to} 
 \prod _{\sigma \in \Gal (k'/{\F_p})} \F \ ; \ 
 a \otimes b \mapsto \bigl( \sigma (a) b\bigr)_{\sigma}
\]
and let $\epsilon_0 \in k' \otimes _{\F_p} \F $ be 
the primitive idempotent corresponding to $\id _{k'}$. 
For $0 \leq r \leq 2n-1$, we put $\epsilon _r =\phi ^r \epsilon _0$.
Note that 
$(a^{p^r} \otimes 1)\epsilon _r =(1 \otimes a) \epsilon _r $ 
for all $a\in k'$.

We put
\begin{multline*}
 e_1 =w_{\lambda} ^{-1} \bigl{\{}
 (u_{2n} ^s \otimes 1)(\epsilon _0 v_1 +\epsilon _n v_2)+
 (u_{2n} ^{ps} \otimes 1)(\epsilon _1 v_1 +\epsilon _{n+1} v_2)+ \\
 \cdots +
 (u_{2n} ^{p^{n-1} s} \otimes 1)(\epsilon _{n-1} v_1 +\epsilon _{2n-1} v_2)
 \bigr{\}}, 
\end{multline*}
\begin{multline*}
 e_2 =w_{\lambda} ^{-1} \bigl{\{}
 (u_{2n} ^{p^n s} \otimes 1)(\epsilon _n v_1 +\epsilon _0 v_2)+
 (u_{2n} ^{p^{n+1} s} \otimes 1)(\epsilon _{n+1} v_1 +\epsilon _1 v_2)+ \\
 \cdots +
 (u_{2n} ^{p^{2n-1} s} \otimes 1)(\epsilon _{2n-1} v_1 +\epsilon _{n-1} v_2)
 \bigr{\}}
\end{multline*}
in 
$(\mathcal{O}_{\mathcal{E}^{\mathrm{ur}}} 
/p\mathcal{O}_{\mathcal{E}^{\mathrm{ur}}} ) 
\otimes _{\F _p} V_{\F}(-1)$.
Then $e_1$ and $e_2$ are fixed by 
$g \in G_{K' _{\infty}}$ and $\widetilde{\Fr} _q $.
Hence $e_1$, $e_2$ are fixed by $G_{K_{\infty}}$, and
these are a basis of $\Phi {\mathrm{M}}_{\mathcal{O}_{\mathcal{E}} ,\F }$
over 
$\mathcal{O}_{\mathcal{E}}
 \otimes _{\mathbb{Z}_p} \F $.
We put 
$\alpha _{\lambda} = w_{\lambda}/\phi (w_{\lambda})$.
As $\phi (w_{\lambda} )$ satisfies 
the condition determining $w_{\lambda} $,
the element $\alpha _{\lambda}$ of 
$(\overline{k} \otimes _{\F _p} \F )^{\times}$
 is in $(k \otimes _{\F _p } \F )^{\times}$.
Now we have 
\begin{align*}
 \phi (e_1 ) &=
 \alpha _{\lambda} \bigl{\{}
 (\epsilon _1 + \epsilon _{n+1} ) + \cdots +
 (\epsilon _{n-1} + \epsilon _{2n-1} )
 \bigr{\}}e_1 +
 \alpha _{\lambda}
 (\epsilon _0 + \epsilon _n )e_2 , 
 \\
 \phi (e_2 ) &=
 \alpha _{\lambda} u^s
 (\epsilon _0 + \epsilon _n )e_1 +
 \alpha _{\lambda} \bigl{\{}
 (\epsilon _1 + \epsilon _{n+1} ) + \cdots +
 (\epsilon _{n-1} + \epsilon _{2n-1} )
 \bigr{\}}e_2.
\end{align*}
If we put
\[
 \sigma _1 =\phi ,\ \sigma _2 = \id _k ,\ 
 \sigma _3 = \phi ^{-1} , 
 \ldots ,\ \sigma _n = \phi ^2, 
\]
then we have 
\[
 \epsilon _{\sigma _1 } =\epsilon _{n-1} + \epsilon _{2n-1} ,\ 
 \epsilon _{\sigma _2 } =\epsilon _0 + \epsilon _n ,
 \ldots ,\ 
 \epsilon _{\sigma _n } =\epsilon _{n-2} + \epsilon _{2n-2} 
\] 
and
\[
 M_{\F} \sim 
\Biggl(
\begin{pmatrix}
 0 & \alpha _1 \\ \alpha _1 u^s & 0
\end{pmatrix}
,
\begin{pmatrix}
 \alpha _2 & 0 \\ 0 & \alpha _2
\end{pmatrix}
,
\ldots
,
\begin{pmatrix}
 \alpha _n & 0 \\ 0 & \alpha _n
\end{pmatrix}
\Biggr). 
\]
Here $\alpha _i $ is the 
$\sigma _{i+1}$-th component of $\alpha _{\lambda}$ 
in $\prod _{\sigma \in \Gal (k/{\F_p})} \F$.
\end{proof} 

\section{Main theorem}

\begin{lem}\label{condition}
If $\F'$ is a finite extension of $\F$, 
the elements of
$\mathscr{GR}^{\mathbf{v}}_{V_{\F},0}(\F')$ 
naturally correspond to free
$k[[u]] \otimes _{\F_p} \F'$-submodules 
$\mathfrak{M}_{\F'}\subset M_\F \otimes _{\F} \F'$ 
of rank $2$ that satisfy the following:
\begin{enumerate}
\item
$\mathfrak{M}_{\F'}$ is $\phi$-stable. 
\item 
For some (so any) choice of
$k[[u]] \otimes _{\F_p} \F'$-basis for 
$\mathfrak{M}_{\F'}$, and for each $\sigma \in \Gal (k/\F_p)$, 
the map
\[
 \phi : \epsilon _{\sigma}\mathfrak{M}_{\F'}\to 
 \epsilon _{\sigma \circ \phi ^{-1}}\mathfrak{M}_{\F'}
\]
has determinant $\alpha u^e$ for some 
$\alpha \in \F'[[u]]^{\times}$.
\end{enumerate}
\end{lem}

\begin{proof}
 This is \cite[Lemma 2.2]{Gee}.
\end{proof}

\begin{lem}\label{connection}
Suppose $x_1 ,x_2 \in \mathscr{GR}^{\mathbf{v}}_{V_{\F},0}(\F)$ 
correspond to objects 
$\mathfrak{M}_{1,\F} , \mathfrak{M}_{2,\F}$ of 
$({\mathrm{Mod}}/\mathfrak{S})_{\F}$ respectively. 
Let $N = (N_i )_{1 \leq i \leq n}$ be a nilpotent element of 
$M_2 \bigl( \F((u))\bigr)^n$
such that $\mathfrak{M}_{2,\F}=(1+N)\cdot \mathfrak{M} _{1,\F}$, 
and $A = (A_i )_{1 \leq i \leq n}$ be an element of 
$GL_2 \bigl( \F((u))\bigr)^n$
such that $\mathfrak{M} _{1,\F} \sim A$. 
If 
$\phi (N_i )A_i N_{i+1} \in M_2\bigl( \F[[u]]\bigr)$ 
for all $i$, then there is a morphism 
$\mathbb{P}^1 \to \mathscr{GR}^{\mathbf{v}} _{V_{\F},0}$ 
sending $0$ to $x_1$ and $1$ to $x_2$.
\end{lem}

\begin{proof}
 This is \cite[Lemma 2.4]{Gee}.
\end{proof}

\begin{lem}\label{operation}
Suppose $n \geq 2$.
Let $\mathfrak{M}_{\F}$ be the object of 
$({\mathrm{Mod}}/\mathfrak{S})_{\F}$ corresponding to 
a point $x \in \mathscr{GR}^{\mathbf{v}}_{V_{\F},0}(\F)$. 
Fix a basis of $\mathfrak{M}_{\F}$ over 
$k[[u]] \otimes _{\F_p} \F$. 
Consider 
$U^{(i)} = (U^{(i)} _j )_{1 \leq j \leq n}
 \in GL _2 \bigl(\F((u))\bigr)^n$ 
such that 
$U^{(i)} _i = 
\begin{pmatrix}
u & 0 \\ 0 & u^{-1}
\end{pmatrix}$ 
and 
$U^{(i)} _j = 
\begin{pmatrix}
1 & 0 \\ 0 & 1
\end{pmatrix}$ 
for all $j \neq i$. 
If $U^{(i)} \cdot \mathfrak{M}_{\F}$ is $\phi$-stable, 
it corresponds to a point 
$x' \in \mathscr{GR}^{\mathbf{v}}_{V_{\F},0}(\F)$, 
and $x'$ lies on the same connected component of
$\mathscr{GR}^{\mathbf{v}}_{V_{\F},0}$ as $x$. 
\end{lem}

\begin{proof}
 First, $U^{(i)} \cdot \mathfrak{M}_{\F}$ corresponds to a point 
$x' \in \mathscr{GR}^{\mathbf{v}}_{V_{\F},0}(\F)$, 
because it satisfies the conditions of Lemma \ref{condition}. 

Next, we consider 
$N^{(i)} =(N^{(i)} _j )_{1 \leq j \leq n} \in M_2 \bigl( \F((u))\bigr)^n$ 
such that 
\[
N^{(i)} _i = 
\begin{pmatrix}
 1 & -u \\ u^{-1} & -1
\end{pmatrix} 
\textrm{ and } N^{(i)} _j = 0 \textrm{ for all } j \neq i. 
\]
Then 
$U^{(i)} \cdot \mathfrak{M}_{\F} = 
 (1+N^{(i)})\cdot \mathfrak{M}_{\F}$, 
because 
$\begin{pmatrix}
  u^{-1} & 0 \\ 0 & u
 \end{pmatrix}
 =
 \begin{pmatrix}
  0 & 1 \\ -1 & 2u
 \end{pmatrix}
 \begin{pmatrix}
  2 & -u \\ u^{-1} & 0 
 \end{pmatrix}$. 
So we can apply Lemma \ref{connection}.
\end{proof}

\begin{thm}\label{main}
Let $\F'$ be a finite extension of $\F$. 
Suppose $x_1 ,x_2 \in \mathscr{GR}^{\mathbf{v}}_{V_{\F},0}(\F')$ 
correspond to objects 
$\mathfrak{M}_{1,\F'} , \mathfrak{M}_{2,\F'}$ of 
$({\mathrm{Mod}}/\mathfrak{S})_{\F'}$ respectively.
If $\mathfrak{M}_{1,\F'}$ and $\mathfrak{M}_{2,\F'}$
 are both non-ordinary, then $x_1$ and $x_2$ lie on the same connected
 component of
 $\mathscr{GR}^{\mathbf{v}}_{V_{\F},0}$.
\end{thm}

\begin{proof}
When $n=1$, this was proved in \cite{Kis}.
If $e<p-1$, then 
$\mathscr{GR}^{\mathbf{v}}_{V_{\F},0}(\F')$ 
is one point by \cite[Theorem 3.3.3]{Ray}. 
So we may assume $n \geq 2$ and $e \geq p-1$. 
Furthermore, replacing $V_{\F}$ by $V_{\F} \otimes _{\F} \F'$, 
we may assume $\F =\F'$.

Suppose first that $V_{\F}$ is reducible.
We can choose a basis so that
$\mathfrak{M}_{1,\F} \sim 
A =(A_i )_{1 \leq i \leq n} \in M_2 \bigl( \F[[u]]\bigr)^n$ 
where 
$A_i =
\begin{pmatrix}
a_i & b_i \\ 0 & c_i
\end{pmatrix}$ 
for $a_i ,b_i ,c_i \in \F[[u]]$, 
because $M_{\F}$ is reducible and 
$\mathfrak{M}_{1,\F}$ is $\phi$-stable. 
By the Iwasawa decomposition and the determinant conditions, 
we can take 
$B=(B_i )_{1 \leq i \leq n} \in GL_2 \bigl( \F((u))\bigr)^n$ 
such that 
$\mathfrak{M}_{2,\F} = B \cdot \mathfrak{M}_{1,\F}$ 
and 
$B_i = 
 \begin{pmatrix}
 u^{-s_i} & v_i \\ 0 & u^{s_i}
 \end{pmatrix}$ 
for $s_i \in \mathbb{Z}$ and 
$v_i \in \F((u))$. 
Then 
$\mathfrak{M}_{2,\F} \sim 
\bigl( \phi (B_i)A_i B_{i+1} ^{-1} \bigr)_i$, 
and we have 
\begin{align*}
 \phi (B_i) A_i B_{i+1} ^{-1}
  & = 
 \begin{pmatrix}
  u^{-ps_i} & \phi (v_i) \\ 0 & u^{ps_i}
 \end{pmatrix}
 \begin{pmatrix}
  a_i & b_i \\ 0 & c_i
 \end{pmatrix}
 \begin{pmatrix}
  u^{s_{i+1}} & -v_{i+1} \\ 0 & u^{-s_{i+1}}
 \end{pmatrix} \\
 &=
 \begin{pmatrix}
  a_i u^{-ps_i +s_{i+1}} &
  -a_i v_{i+1} u^{-ps_i} +b_i u^{-ps_i -s_{i+1}}+c_i \phi (v_i) u^{-s_{i+1}} \\
  0 & c_i u^{ps_i -s_{i+1}}
 \end{pmatrix}.
\end{align*}
In the last matrix, every component is integral
because $\mathfrak{M}_{2,\F}$ is $\phi$-stable. 

First of all, we want to reduce the problem to the case where 
$s_i =0$ for all $i$. 
When $e=p-1$, we have 
$0 \leq v_u (c_i) \leq p-1$ and 
$0 \leq v_u (c_i)+ps_i -s_{i+1} \leq p-1$ 
for all $i$ by the determinant conditions. From the second set 
of inequalities, we obtaine 
\[
 0 \leq 
 \sum _{j=0} ^{n-1} 
 \bigl\{ v_u (c_{i-1-j}) + ps_{i-1-j} -s_{i-j} \bigr\}
 p^j 
 \leq p^n -1, 
\] 
and we have 
\[
 \sum _{j=0} ^{n-1} 
 \bigl\{ v_u (c_{i-1-j}) + ps_{i-1-j} -s_{i-j} \bigr\}
 p^j =
 (p^n -1)s_i +
 \sum _{j=0} ^{n-1} v_u (c_{i-1-j}) p^j .
\]
Combining these with $0 \leq v_u (c_i) \leq p-1$, 
we get $-1 \leq s_i \leq 1$. 
If $s_i =1$ for some $i$, 
the second sign of the above inequality must be 
the equality sign. 
So we get $v_u (c_j)=0$ for all $j$. 
This contradicts the non-ordinarity of 
$\mathfrak{M}_{1,\F}$. 
If $s_i =-1$ for some $i$, 
the first sign of the above inequality must be 
the equality sign. 
So we get $v_u (c_j) +ps_j -s_{j+1} =0$ for all $j$. 
This contradicts the non-ordinarity of 
$\mathfrak{M}_{2,\F}$. 
Hence, we have $s_i =0$ for all $i$. 
So we may assume $e \geq p$. 

We consider $U^{(i)}$ as in Lemma \ref{operation}. 
If $s_i >0$ and $U^{(i)} \cdot \mathfrak{M}_{2,\F}$ 
is $\phi$-stable, 
we may replace $\mathfrak{M}_{2,\F}$ with
$U^{(i)} \cdot \mathfrak{M}_{2,\F}$ by Lemma \ref{operation}. 
This replacement changes $s_i$ into $s_i -1$ and 
$v_i$ into $uv_i$. 
If $s_i <0$, switching $\mathfrak{M}_{1,\F}$ with 
$\mathfrak{M}_{2,\F}$ so that we have $s_i >0$, 
we consider the same replacement as above. 
Note that these replacements decrease 
$|s_i|$ by $1$.
We prove that we can continue these replacements 
until we get to the case where $s_i =0$ for all $i$. 
Suppose that we cannot continue the replacements 
and there is some nonzero $s_i$.
Take an index $i_0$ such that $|s_{i_0}|$ is the greatest.
By switching $\mathfrak{M}_{1,\F}$ with $\mathfrak{M}_{2,\F}$, 
we may assume $s_{i_0} >0$. 
As we cannot continue the replacements, 
we cannot decrease $s_{i_0}$ keeping the 
$\phi$-stability, that is, 
\[
 v_u (c_{i_0}) + ps_{i_0} -s_{{i_0}+1}\leq p-1 
 \textrm{ or } 
 v_u (a_{{i_0}-1}) -ps_{{i_0}-1} +s_{i_0} =0.
\]

If $v_u (c_{i_0}) + ps_{i_0} -s_{{i_0}+1}\leq p-1$, 
we have $s_{i_0} =1$, $v_u (c_{i_0})=0$ and $s_{{i_0}+1}=1$, 
because 
$v_u (c_{i_0}) + (p-1)s_{i_0} +(s_{i_0} -s_{{i_0}+1}) \leq p-1$. 
Now we have $v_u (a_{i_0}) -ps_{i_0} +s_{{i_0}+1}\geq 1$, because 
$e\geq p$ and $v_u (c_{i_0}) + ps_{i_0} -s_{{i_0}+1}\leq p-1$.
As $s_{{i_0}+1}$ cannot be decreased, 
$v_u (c_{{i_0}+1}) + ps_{{i_0}+1} -s_{{i_0}+2}\leq p-1$.
The same argument shows that $v_u (c_i)=0$ and 
$s_i =1$ for all $i$. 
This contradicts the non-ordinarity 
of $\mathfrak{M}_{1,\F}$.

If $v_u (a_{{i_0}-1}) -ps_{{i_0}-1} +s_{i_0} =0$, 
then $s_{{i_0}-1} >0$ 
and $v_u (c_{{i_0}-1}) +ps_{{i_0}-1} -s_{i_0} =e \geq p$. 
As $s_{{i_0}-1}$ cannot be decreased, 
$v_u (a_{{i_0}-2}) -ps_{{i_0}-2} +s_{{i_0}-1} =0$.
The same argument shows that 
$v_u (a_i) -ps_i +s_{i+1} =0$ for all $i$. 
So we have that 
$\mathfrak{M}_{2,\F}$ is 
an extension of a multiplicative module 
by an \'etale module. 
We show that such an extension splits. 
Now we have 
$\mathfrak{M}_{2,\F} \sim 
 \Biggl(
 \begin{pmatrix}
 a_i ' & b_i ' \\
 0 & u^e c_i '
 \end{pmatrix}
 \Biggr)_i$ 
for $a_i ',c_i ' \in \F[[u]]^{\times}$ and 
$b_i ' \in \F[[u]]$. 
Then 
\[ 
 \Biggl(
 \begin{pmatrix}
 1 & v_i ' \\
 0 & 1 
 \end{pmatrix}
 \Biggr)_i \cdot 
 \mathfrak{M}_{2,\F} \sim 
 \Biggl(
 \begin{pmatrix}
 a_i ' & -a_i ' v_{i+1} ' +b_i '+u^e c_i '\phi(v_i ') \\
 0 & u^e c_i '
 \end{pmatrix}
 \Biggr)_i
\]
for $v_i '\in \F[[u]]$. 
It suffices to show that 
there is $(v_i ')_{1\leq i \leq n} \in \F[[u]]^n$ 
such that 
$a_i ' v_{i+1} ' =b_i ' +u^e c_i '\phi(v_i ')$ 
for all $i$, 
and we can solve the system of equations 
by finding $v_i '$ successively 
in ascending order of their degrees. 
Hence we have that 
$\mathfrak{M}_{2,\F}$ is ordinary, 
and this is a contradiction. 

Thus we may assume $s_i =0$ for all $i$. 
Consider 
$N=(N_i )_{1 \leq i \leq n} \in M_2 \bigl( \F((u))\bigr)^n$ 
such that 
$N_i =
\begin{pmatrix}
  0 & v_i \\ 0 & 0
\end{pmatrix}$ 
for $v_i \in \F((u))$. 
Then we have 
$\mathfrak{M}_{2,\F}=(1+N)\cdot \mathfrak{M}_{1,\F}$ 
and 
$\phi (N_i)
\begin{pmatrix}
 a_i & b_i \\ 0 & c_i 
\end{pmatrix} N_{i+1}=0$.
Hence $x_1$ and $x_2$ lie on the same 
connected component by Lemma \ref{connection}. 
This completes the proof in the case where 
$V_{\F}$ is reducible. 

 From now on, we consider the case where 
$V_{\F}$ is irreducible.
If $V_{\F}$ is reducible after extending the base field $\F$, 
we can reduce this case to the reducible case.
So we may assume $V_{\F}$ is absolutely irreducible.
Extending the field $\F$, we have
\[
 M_{\F} \sim 
\Biggl(
\alpha _1
\begin{pmatrix}
 0 & 1 \\ u^s & 0
\end{pmatrix}
,
\alpha _2
\begin{pmatrix}
 1 & 0 \\ 0 & 1
\end{pmatrix}
,
\ldots
,
\alpha _n
\begin{pmatrix}
 1 & 0 \\ 0 & 1
\end{pmatrix}
\Biggr)
\]
for some $\alpha _i \in \F^{\times}$ and 
a positive integer $s$ by Lemma \ref{stracture}. 
This basis gives a sublattice $\mathfrak{M}_{\F}$.
By the Iwasawa decomposition, we can take 
$s_i ' ,t_i ' \in \mathbb{Z}$ and 
$v_i ' \in \F((u))$ so that 
$\mathfrak{M}_{1,\F} =
 \Biggl( 
 \begin{pmatrix}
 u^{s_i '} & v_i ' \\ 0 & u^{t_i '}
 \end{pmatrix}
 \Biggr)_i \cdot 
 \mathfrak{M}_{\F}$. 
Changing the basis by 
$\Biggl( 
 \begin{pmatrix}
 u^{s_i '} & 0 \\ 0 & u^{t_i '}
 \end{pmatrix} 
 \Biggr)_i$, 
we get
\[
 M_{\F} \sim 
\Biggl(
\alpha _1
\begin{pmatrix}
 0 & u^{s_1} \\ u^{t_1} & 0
\end{pmatrix}
,
\alpha _2
\begin{pmatrix}
 u^{s_2} & 0 \\ 0 & u^{t_2}
\end{pmatrix}
,
\ldots
,
\alpha _n
\begin{pmatrix}
 u^{s_n} & 0 \\ 0 & u^{t_n}
\end{pmatrix}
\Biggr).
\]
Here we have $0 \leq t_1$, 
$0 \leq s_i , t_i \leq e$ for 
$2\leq i \leq n$, 
and $s_i + t_i =e$ for all $i$ 
by the $\phi$-stability and the determinant conditions of 
$\mathfrak{M}_{1,\F}$. 

We are going to change the basis so that 
we have moreover $t_1 \leq e$. 
Changing the basis of the $i$-th component by 
$
\begin{pmatrix}
 u & 0 \\ 0 & u^{-1}
\end{pmatrix}$, 
we get the following transformations: 
\begin{align*}
 T_i &:t_i \leadsto t_i -p,\ t_{i-1} \leadsto t_{i-1} +1 
 \ \textrm{for}\ i\neq 2, 
\\ T_2 &: t_2 \leadsto t_2 -p,\ t_1 \leadsto t_1 -1. 
\end{align*}
If $t_1>e$, we put 
\[
 m=\max \{\ 1\leq i \leq n \ |\ t_i \neq e\ \}, 
\]
and carry out $T_1$ when $m=n$, 
and $T_{m+1} , T_{m+2} , \ldots , T_n , T_1$ 
when $m \neq n$.
Then $0\leq s_i , t_i \leq e$ for 
$2\leq i \leq n$, and 
$t_1$ decrease by $p$ when $m \neq 1$, 
by $p+1$ when $m=1$.
Repeat this until we get to 
the situation where $t_1 \leq e$.
If $e\geq p$, we get to 
the situation where $0\leq s_1 ,t_1 \leq e$.
If $e=p-1$ and we do not get to 
the situation where $0\leq s_1 ,t_1 \leq p-1$, 
then we have 
\[
M_{\F} \sim
\Biggl(
\alpha _1
\begin{pmatrix}
 0 & u^{-1} \\ u^p & 0
\end{pmatrix}
,
\alpha _2
\begin{pmatrix}
 1 & 0 \\ 0 & u^{p-1}
\end{pmatrix}
,
\ldots
,
\alpha _n
\begin{pmatrix}
 1 & 0 \\ 0 & u^{p-1}
\end{pmatrix}
\Biggr).
\] 
In this case, changing the basis by
$
\Biggl(
\begin{pmatrix}
 1 & u^{-1} \\ 0 & 1
\end{pmatrix} 
\Biggr)_i$, 
we get 
\[
M_{\F} \sim
 \Biggl(
\alpha _1
\begin{pmatrix}
 1 & 0 \\ u^p & -u^{p-1}
\end{pmatrix}
,
\alpha _2
\begin{pmatrix}
 1 & 0 \\ 0 & u^{p-1}
\end{pmatrix}
,
\ldots
,
\alpha _n
\begin{pmatrix}
 1 & 0 \\ 0 & u^{p-1}
\end{pmatrix}
\Biggr).
\]
This contradicts that $M_{\F}$ is irreducible.
Hence we obtain a basis such that
\[
 M_{\F} \sim 
\Biggl(
\alpha _1
\begin{pmatrix}
 0 & u^{s_1} \\ u^{t_1} & 0
\end{pmatrix}
,
\alpha _2
\begin{pmatrix}
 u^{s_2} & 0 \\ 0 & u^{t_2}
\end{pmatrix}
,
\ldots
,
\alpha _n
\begin{pmatrix}
 u^{s_n} & 0 \\ 0 & u^{t_n}
\end{pmatrix}
\Biggr)
\]
for some $s_i$ and $t_i$ satisfying $s_i + t_i =e$ 
and $0 \leq s_i ,t_i \leq e$
for all $i$.
Let $\mathfrak{M}_{0,\F}$ be the sublattice of 
$M_{\F}$ determined by this basis.
Note that $\mathfrak{M}_{0,\F}$ satisfies 
the conditions of Lemma \ref{condition}, 
and let $x_0$ be the point of 
$\mathscr{GR}^{\mathbf{v}}_{V_{\F},0}$ 
corresponding to $\mathfrak{M}_{0,\F}$.

We prove that we can change 
$(t_i )_{1 \leq i \leq n}$ furthermore 
by $T_i$'s or $T_i ^{-1}$'s 
keeping $0 \leq t_i \leq e$ for all $i$, 
and get to the situation where 
$|s_i  -t_i |\leq p+1$ for all $i$. 
By Lemma \ref{operation}, these changes do not affect 
which of the connected components $x_0$ lies on. 
If $e\leq p+1$, this is satisfied automatically.
So we may assume $e\geq p+2$.
We prove that 
if there is an index $j$ such that 
$|s_j -t_j |\geq p+2$, then there is an index $j_0$ 
such that $|s_{j_0} -t_{j_0} |\geq p+2$ and 
we can change $t_{j_0}$ by $T_{j_0}$ or $T_{j_0} ^{-1}$
so that $|s_{j_0} -t_{j_0} |$ decreases 
keeping $0 \leq t_i \leq e$ for all $i$.
We put 
$h_i =(-1)^{[(i-2)/n]} (s_i -t_i)$ 
for $i \in \mathbb{Z}$. 
By assumption, there is an integer 
$j_0$ such that 
$1 \leq j_0 \leq 2n$, 
$h_{j_0} \geq p+2$ and $h_{j_0 -1} <e$. 
If $2\leq j_0 \leq n+1$, 
we can change $t_{j_0}$ by 
$T_{j_0} ^{-1}$, 
otherwise by $T_{j_0}$, 
so that $|s_{j_0} -t_{j_0} |$ decreases 
keeping $0 \leq t_i \leq e$ for all $i$.
Thus we have proved the claim.
Hence if 
$|s_j -t_j |\geq p+2$ for an index $j$, 
we can carry out 
$T_{j_0}$ or $T_{j_0} ^{-1}$ for an index $j_0$ as above 
and this operation decreases 
$\sum _{i=1} ^n |s_i -t_i |$ by at least $2$.
So after finitely many operations, we get to 
the situation where $|s_i -t_i |\leq p+1$ for all $i$.

Hence we may assume that $s_i$ and $t_i$ satisfy 
$s_i + t_i =e$, $0 \leq s_i ,t_i \leq e$ 
and $|s_i -t_i| \leq p+1$ for all $i$. 
We are going to prove that 
$x_0$ and $x_1$ lie on the same connected component. 
We can prove that 
$x_0$ and $x_2$ lie on the same connected component 
by the same argument.
 
By the Iwasawa decomposition and the determinant conditions, 
we can take 
$B=(B_i )_{1 \leq i \leq n} \in GL_2 \bigl( \F((u))\bigr)^n$ 
such that 
$\mathfrak{M}_{1,\F}=B \cdot \mathfrak{M}_{0,\F}$ 
and 
$B_i=
\begin{pmatrix}
 u^{-a_i} & v_i \\ 0 & u^{a_i}
\end{pmatrix}$ 
for $a_i \in \mathbb{Z}$ 
and $v_i \in \F((u))$. 
Then we put $r_i =v_u (v_i)$. 
Now we have
\begin{align*}
 \phi (B_1 )
 \begin{pmatrix}
  0 & u^{s_1} \\ u^{t_1} & 0
 \end{pmatrix}
 B_2 ^{-1}
 &=
 \begin{pmatrix}
  \phi (v_1 ) u^{t_1 +a_2} & u^{s_1 -pa_1 -a_2} - \phi (v_1 )v_2 u^{t_1}
  \\ u^{t_1 + pa_1 +a_2} & -v_2 u^{t_1 +pa_1}
 \end{pmatrix},\\
 \phi (B_i )
 \begin{pmatrix}
  u^{s_i} & 0 \\ 0 & u^{t_i}
 \end{pmatrix}
 B_{i+1} ^{-1}
 &=
 \begin{pmatrix}
  u^{s_i -pa_i +a_{i+1}} & \phi (v_i )u^{t_i -a_{i+1}} - v_{i+1} u^{s_i -pa_i}
  \\ 0 & u^{t_i +pa_i -a_{i+1}}
 \end{pmatrix}
\end{align*}
for $2\leq i \leq n$. 
On the right-hand sides, every component of the matrices 
is integral because $\mathfrak{M}_{1,\F}$ is $\phi$-stable. 

First, we consider the case 
$t_1 + pa_1 +a_2 >e$.
In this case, 
\[
 (pr_1 + t_1 +a_2)+(r_2 +t_1 +pa_1)=e
 , \ 
 s_1-pa_1 -a_2 = pr_1 +r_2 +t_1 < 0
\]
by the $\phi$-stability and the determinant conditions of 
$\mathfrak{M}_{1,\F}$. 
We have $a_1 > r_1$, because 
$t_1 +pa_1 +a_2 >e \geq pr_1 +t_1 +a_2$.
Similarly, we have $a_2 > r_2 $, because 
$t_1 +pa_1 +a_2 >e \geq r_2 +t_1 +pa_1$.

We consider the following operations: 
\[
 a_i \leadsto a_i -1,\ v_i \leadsto uv_i ,\ 
 \textrm{if it preserves the}\ \phi \textrm{-stability of}\ 
 B\cdot \mathfrak{M}_{0,\F}.
\]
These operations replace $x_1$ by a point 
that lies on the same connected component as $x_1$ 
by Lemma \ref{operation}.
We prove that we can continue these operations until we get to 
the situation where $t_1 +pa_1 +a_2 \leq e$. 
In other words, 
we reduce the problem to the case $t_1 +pa_1 +a_2 \leq e$. 
If we can continue the operations endlessly, 
we get to the situation where 
$t_1 +pa_1 +a_2 \leq e$, 
because the conditions $s_i -pa_i +a_{i+1} \geq 0$ for 
$2\leq i \leq n$ exclude 
that both $a_1$ and $a_2$ remain bounded below. 
Suppose we cannot continue the operations.
This is equivalent to the following condition:
\begin{align*}
 s_n -pa_n +a_1 &=0\ {\textrm{or}}\ r_2 +t_1 +pa_1 \leq p-1, \\
 pr_1 +t_1 +a_2 &=0\ {\textrm{or}}\ t_2 +pa_2 -a_3 \leq p-1, \\
 s_{i-1} -pa_{i-1} +a_i &=0\ {\textrm{or}}\ t_i +pa_i -a_{i+1} \leq p-1 
 \ \textrm{for each}\ 3\leq i \leq n. 
\end{align*}
If $e\geq p$, there are only the following two cases, 
because $(pr_1 + t_1 +a_2)+(r_2 +t_1 +pa_1)=e$ and 
$(s_i -pa_i +a_{i+1})+(t_i +pa_i -a_{i+1})=e$ for $2\leq i \leq n$.
\begin{align*}
 &\textrm{Case\,1}:pr_1 +t_1 +a_2=0,\ s_i -pa_i +a_{i+1}=0\ 
 \textrm{for}\ 2\leq i \leq n. \\
 &\textrm{Case\,2}:r_2 +t_1 +pa_1\leq p-1,\ t_i +pa_i -a_{i+1} \leq p-1\ 
 \textrm{for}\ 2\leq i \leq n.
\end{align*}
If $e=p-1$, clearly it is in Case\,2.

In the Case\,1.
Suppose that there is an index $i$ such that 
$2 \leq i \leq n$ and 
$pr_i +t_i -a_{i+1} \neq r_{i+1} +s_i -pa_i$.
Then both sides are non-negative,
because 
$v_u (\phi (v_i )u^{t_i -a_{i+1}} - v_{i+1} u^{s_i -pa_i})\geq 0$.
Comparing $r_{i+1} +s_i -pa_i \geq 0$ with $s_i -pa_i +a_{i+1}=0$, we get 
$r_{i+1} \geq a_{i+1}$.
Then 
$pr_{i+1} +t_{i+1} -a_{i+2} \geq pa_{i+1} + t_{i+1} -a_{i+2}\geq 0$, 
and $r_{i+2} +s_{i+1} -pa_{i+1} \geq 0$ 
because 
$v_u (\phi (v_{i+1})u^{t_{i+1} -a_{i+2}} -v_{i+2} u^{s_{i+1} -pa_{i+1}})\geq 0$. 
Comparing $r_{i+2} +s_{i+1} -pa_{i+1} \geq 0$ 
with $s_{i+1} -pa_{i+1} +a_{i+2}=0$, 
we get $r_{i+2} \geq a_{i+2}$. 
The same argument goes on and shows $r_1 \geq a_1$.
This is a contradiction.
Thus $pr_i +t_i -a_{i+1} = r_{i+1} +s_i -pa_i$ for all 
$2 \leq i \leq n$. 
Now we change the basis of 
\[
 M_{\F} \sim 
\Biggl(
\alpha _1
\begin{pmatrix}
 0 & u^{s_1} \\ u^{t_1} & 0
\end{pmatrix}
,
\alpha _2
\begin{pmatrix}
 u^{s_2} & 0 \\ 0 & u^{t_2}
\end{pmatrix}
,
\ldots
,
\alpha _n
\begin{pmatrix}
 u^{s_n} & 0 \\ 0 & u^{t_n}
\end{pmatrix}
\Biggr)
\]
by 
$
\Biggl(
\begin{pmatrix}
 u^{-a_i} & u^{r_i} \\ 0 & u^{a_i}
\end{pmatrix}
\Biggr)_i$. 
Then we have 
\[
 M_{\F} \sim 
\Biggl(
\alpha _1
\begin{pmatrix}
 1 & 0 \\ u^{t_1 +pa_1 +a_2} & -u^e
\end{pmatrix}
,
\alpha _2
\begin{pmatrix}
 1 & 0 \\ 0 & u^e
\end{pmatrix}
,
\ldots
,
\alpha _n
\begin{pmatrix}
 1 & 0 \\ 0 & u^e
\end{pmatrix}
\Biggr),
\]
and this contradicts that $M_{\F}$ is irreducible.

In the Case\,2.
Suppose that there is an index $i$ such that $2\leq i \leq n$ and 
$pr_i +t_i -a_{i+1} \neq r_{i+1} + s_i -pa_i$. 
Then both sides are non-negative,
because 
$v_u (\phi (v_i )u^{t_i -a_{i+1}} - v_{i+1} u^{s_i -pa_i})\geq 0$.
Comparing $pr_i +t_i -a_{i+1} \geq 0$ with $t_i +pa_i -a_{i+1} \leq p-1$, 
we get $r_i \geq a_i$. Then 
$r_i +s_{i-1} -pa_{i-1} \geq s_{i-1} -pa_{i-1} + a_i \geq 0$, 
and $pr_{i-1} + t_{i-1} -a_i \geq 0$ because 
 $v_u (\phi (v_{i-1})u^{t_{i-1} -a_i} 
 - v_i u^{s_{i-1} -pa_{i-1}})\geq 0$. 
Comparing $pr_{i-1} +t_{i-1} -a_i \geq 0$ with 
$t_{i-1} +pa_{i-1} -a_i \leq p-1$, 
we get $r_{i-1} \geq a_{i-1}$. 
The same argument goes on 
and shows that $r_2 \geq a_2$. This is a contradiction.

The above argument shows that 
\[
 r_i <a_i,\ 
 pr_i +t_i - a_{i+1}=r_{i+1} +s_i -pa_i <0\ 
 \textrm{for} \ 2\leq i \leq n.
\]
Combining these equations with
$s_1 -pa_1 -a_2 = pr_1 +r_2 +t_1$, we get
\begin{multline*}
 -(p^n +1)r_1 = (p^n +1)a_1 +(s_n -t_n)+p(s_{n-1} - t_{n-1})+\\
 \cdots +p^{n-3}(s_3 -t _3)+p^{n-2}(s_2 -t_2)-p^{n-1}(s_1 -t_1),
\end{multline*}
\begin{multline*}
 -(p^n +1)r_2 = (p^n +1)a_2 -(s_1 -t_1)-p(s_n - t_n)-\\
 \cdots -p^{n-3}(s_4 -t _4)-p^{n-2}(s_3 -t_3)-p^{n-1}(s_2 -t_2),
\end{multline*}
\begin{multline*}
 -(p^n +1)r_3 = (p^n +1)a_3 +(s_2 -t_2)-p(s_1 - t_1)-\\
 \cdots -p^{n-3}(s_5 -t _5)-p^{n-2}(s_4 -t_4)-p^{n-1}(s_3 -t_3),
\end{multline*}
\hspace*{6.9em} $\vdots$
\begin{multline*}
 -(p^n +1)r_n = (p^n +1)a_n +(s_{n-1} -t_{n-1})+p(s_{n-2} - t_{n-2})+\\
 \cdots +p^{n-3}(s_2 -t _2)-p^{n-2}(s_1 -t_1)-p^{n-1}(s_n -t_n).
\end{multline*}
As $|s_i -t_i|\leq p+1$ and
\[
 (p+1)+p(p+1)+\cdots +p^{n-1}(p+1)=\biggl(\frac{p^n -1}{p-1}\biggr)(p+1)<2(p^n +1),
\] 
we get $-a_i -1\leq r_i \leq -a_i +1$.
When $e=p-1$, as $|s_i -t_i|\leq p-1$ and
\[
 (p-1)+p(p-1)+\cdots +p^{n-1}(p-1)=\biggl(\frac{p^n -1}{p-1}\biggr)(p-1)<(p^n +1), 
\] 
we get $r_i =-a_i$.

As $r_2 +t_1 +pa_1 \leq p-1$, we have 
\[
 pa_1 \leq t_1 +pa_1 \leq p-1-r_2 \leq p+a_2.
\]
For $2\leq i \leq n$, as $t_i +pa_i - a_{i+1}\leq p-1$, we have 
\[
 pa_i \leq t_i + pa_i \leq p-1 +a_{i+1}.
\]
Take an index $i_0$ such that $a_{i_0}$ is the greatest.
As $pa_{i_0} \leq a_{i_0 +1} +p \leq a_{i_0} +p$, 
we get $a_{i_0} \leq \frac{p}{p-1}<2$.
Combining $-a_i -1\leq r_i$ and $r_i <a_i$, 
we get $a_i \geq 0$.
Hence
\[
 a_i =0,\  r_i =-1,\ \textrm{or}
 \ a_i =1,\  -2\leq r_i \leq 0
\]
for every $i$.

In the case $a_2=0$, we have $r_2 =-1$. 
Comparing $t_1 +pa_1 +a_2 >e$ with 
$r_2 + t_1 +pa_1 \leq p-1$, we get $e <p$.
When $e=p-1$, we have $r_2 =-a_2$. This is a contradiction. 

In the case $a_2 =1$.
As $0\leq t_i +pa_i -a_{i+1} \leq p-1$ for $2 \leq i \leq n$, 
we have $a_i =1$ for all $i$ and $t_i =0$ for $2\leq i \leq n$.
As $r_2 +pa_1 +t_1 \leq p-1$, we have $r_2 \leq -1$.
As $pr_2 +t_2 -a_3 =r_3 +s_2 -pa_2$, we have 
$r_3 =pr_2 +p-1-e\leq -e-1\leq -3$.
This is a contradiction.

Thus we may assume $t_1 +pa_1 +a_2 \leq e$.
We put 
$\mathfrak{M}_{3,\F}=
\Biggl(
\begin{pmatrix}
 u^{-a_i} & 0 \\ 0 & u^{a_i}
\end{pmatrix} 
\Biggr)_i
\cdot 
\mathfrak{M}_{0,\F}$, 
then  
\begin{multline*}
\mathfrak{M}_{3,\F} \sim 
\Biggl(
\alpha _1
\begin{pmatrix}
 0 & u^{s_1 -pa_1 -a_2} \\ u^{t_1 +pa_1 +a_2} & 0
\end{pmatrix}
,
\alpha _2
\begin{pmatrix}
 u^{s_2 -pa_2 +a_3} & 0 \\ 0 & u^{t_2 +pa_2 -a_3}
\end{pmatrix}
,\\
\ldots
,
\alpha _n
\begin{pmatrix}
 u^{s_n -pa_n +a_1} & 0 \\ 0 & u^{t_n +pa_n -a_1}
\end{pmatrix}
\Biggr)
\end{multline*}
and 
$\mathfrak{M}_{1,\F}=
\Biggl(
\begin{pmatrix}
 1 & v_i u^{-a_i} \\ 0 & 1
\end{pmatrix}
\Biggr)_i 
\cdot 
\mathfrak{M}_{3,\F}$.
Note that $\mathfrak{M}_{3,\F}$ satisfies 
the conditions of Lemma \ref{condition}, 
and let $x_3$ be the point of 
$\mathscr{GR}^{\mathbf{v}}_{V_{\F},0}$ 
corresponding to $\mathfrak{M}_{3,\F}$. 
If we put 
$N_i =
\begin{pmatrix}
 0 & v_i u^{-a_i} \\ 0 & 0
\end{pmatrix}$, then 
\[
 \phi (N_1)
\begin{pmatrix}
 0 & u^{s_1 -pa_1 -a_2} \\ u^{t_1 +pa_1 +a_2} & 0
\end{pmatrix}
 N_2 =
\begin{pmatrix}
 0 & \phi (v_1 )v_2 u^{t_1} \\ 0 & 0
\end{pmatrix},
\]
\[
 \phi (N_i)
\begin{pmatrix}
 u^{s_i -pa_i + a_{i+1}} & 0 \\ 0 & u^{t_i +pa_i -a_{i+1}}
\end{pmatrix}
N_{i+1} =0
\]
for $2 \leq i \leq n$. 
Here we have 
$v_u \bigl(\phi (v_1 )v_2 u^{t_1}\bigr)\geq 0$, because
$s_1 -pa_1 -a_2 \geq 0$ and 
$v_u \bigl(
u^{s_1 -pa_1 -a_2} - \phi (v_1 )v_2 u^{t_1}
\bigr)\geq 0$.
Hence 
$x_1$ and $x_3$ lie on the same connected component 
by Lemma \ref{connection}. 

We are going to compare 
$\mathfrak{M}_{0,\F}$ and $\mathfrak{M}_{3,\F}$.
Recall the previous operations on the basis of $\mathfrak{M}_{0,\F}$ 
that changed 
$(t_i )_{1 \leq i \leq n}$ 
so that $|s_i -t_i |\leq p+1$ keeping 
$0 \leq t_i \leq e$ for all $i$. 
Apply the same operations to the basis of 
$\mathfrak{M}_{3,\F}$. 
By Lemma \ref{operation}, these operations do not affect 
which of the connected components $x_3$ lies on. 
So we may assume that 
\[
 s_1 -pa_1 -a_2,\ s_2 - pa_2 +a_3, 
 \ldots ,\ s_n-pa_n +a_1 
\] 
are all in $\bigl[ (e-p-1)/2,(e+p+1)/2 \bigr]$.
As $(e-p-1)/2 \leq s_i \leq (e+p+1)/2$, we have that 
\[
 |pa_1 +a_2 |\leq p+1,\ |pa_2 -a_3 |\leq p+1, 
 \ldots ,\ |pa_n - a_1 |\leq p+1.
\] 
Summing up the above inequalities 
after multiplying some $p$-powers 
so that we can eliminate $a_j$ for $j \neq i$, 
we get 
$|(p^n +1)a_i | \leq \bigl\{(p^n -1)/(p-1)\bigr\}(p+1)$. 
So we have $|a_i | \leq 1$ for all $i$. 

In the case $e\geq p$. 
We consider the operations that decrease 
$|a_i|$ by $1$ for an index $i$ 
keeping the condition of $\phi$-stability. 
By Lemma \ref{operation}, 
these operations do not affect 
which of the connected components $x_3$ lies on. 
We prove that we can continue the operations 
until we have $a_i =0$ for all $i$, 
that is, $x_0$ and $x_3$ lie on the 
same connected component. 
Suppose that we cannot continue the operations 
and there is some nonzero $a_i$.
The condition of $\phi$-stability is equivalent to
\begin{multline*}
 C_1 :0\leq s_1 -pa_1 -a_2 \leq e,\ 
 C_2 :0\leq s_2 -pa_2 +a_3 \leq e,\ \\
 \ldots ,\ 
 C_n :0\leq s_n -pa_n +a_1 \leq e.
\end{multline*}
Note that 
if $a_i \neq 0$ or $a_{i+1} \neq 0$, 
we can decrease 
$|a_i|$ or $|a_{i+1}|$ keeping $C_i$.

We put 
\[
 c_i =\sharp\bigl\{ i\leq j\leq i+1 \bigm| 
 \textrm{we can decrease}\ |a_j|\ \textrm{keeping}\ C_i \bigr\},
\] 
and claim that
$\sharp \{ j\ |\ a_j \neq 0 \} =\sum_{i=1}^n c_i$.
First, if $a_i \neq 0$, we have 
$c_{i-1} \geq 1$ and $c_i \geq 1$ from the above remark. 
So we have $\sharp \{ j\ |\ a_j \neq 0 \} \leq \sum_{i=1}^n c_i$. 
Second, we count $a_i \neq 0$ in not both of 
$C_{i-1}$ and $C_i$, because we cannot continue the operations. 
So we have 
$\sharp \{ j\ |\ a_j \neq 0 \} \geq \sum_{i=1}^n c_i$. 
Hence we have equality. From this equality, 
we have $a_i \neq 0$ and $c_i =1$ for all $i$. 
For $2\leq i \leq n$, we have $a_i =a_{i+1}\neq 0$ because $c_i =1$. 
So we have $a_1 =a_2 \neq 0$, but this contradicts $c_1 =1$. 

In the case $e=p-1$. 
We have 
$|pa_1 +a_2 | \leq p-1$ by $C_1$, 
and 
$|pa_i -a_{i+1} |\leq p-1$ by $C_i$ 
for $2 \leq i \leq n$. 
Summing up these inequalities 
after multiplying some $p$-powers 
so that we can eliminate $a_j$ for $j \neq i$, 
we get 
$|(p^n +1)a_i| \leq p^n -1$. 
So we have $a_i =0$ for all $i$. 

Hence $x_0$ and $x_3$ lie on the same connected component.
This completes the proof.  
\end{proof}

\section{Application}

As an application of Theorem \ref{main}, we can improve 
a theorem in \cite{Kis} 
comparing a deformation ring and a Hecke ring. 
We recall some notation from \cite{Kis}, 
and the interested reader should consult \cite{Kis} 
for more detailed definitions.

Let $F$ be a totally real field, 
and $D$ be a totally definite quaternion algebra with center $F$. 
Let $\Sigma$ be the set of finite primes where $D$ is ramified. 
We assume that $\Sigma$ does not contain any primes dividing $p$. 
We put 
$\Sigma _p =\Sigma \cup \{\mathfrak{p}\}_{\mathfrak{p}\mid p}$, 
and fix a maximal order $\mathcal{O}_D$ of $D$. 
Let 
$U=\prod _v U_v \subset (D\otimes _F \mathbb{A}_F ^f)^{\times}$
be a compact open subgroup contained in 
$\prod _v (\mathcal{O}_D )_v ^{\times}$, 
and we assume that $U_v =(\mathcal{O}_D )_v ^{\times}$ 
for all $v \in \Sigma _p$. 
Let $\mathcal{O}$ be the ring of integers of a $p$-adic field.
We fix a continuous character 
$\psi :(\mathbb{A}_F ^f)^{\times} /F^{\times} \to \mathcal{O}^{\times}$ 
such that 
$\psi$ is trivial on $U_v \cap \mathcal{O}_{F_v} ^{\times}$ 
for any finite place $v$ of $F$.
Let $S$ be a finite set of primes containing 
the infinite primes, $\Sigma _p$, 
and the finite primes $v$ of $F$ 
such that $U_v \subset D_v ^{\times}$ 
is not maximal compact. 
We fix a decomposition group 
$G_{F_v} \subset G_{F,S}$ for each $v \in S$.
Let 
$\mathbb{T}' _{\psi ,\mathcal{O}} (U)$ 
(resp. $\mathbb{T}_{\psi ,\mathcal{O}} (U)$) denote 
the image of $\mathbb{T}_{S,\mathcal{O}} ^{\mathrm{univ}}$ 
(resp. $\mathbb{T}_{S^p ,\mathcal{O}} ^{\mathrm{univ}}$)
in the endomorphism ring of $S_{2,\psi} (U,\mathcal{O})$.
Let $\mathfrak{m}$ be a maximal ideal of 
$\mathbb{T}_{\psi ,\mathcal{O}} (U)$
that induces a non-Eisenstein maximal ideal of 
$\mathbb{T}_{S^p ,\mathcal{O}} ^{\mathrm{univ}}$, 
and put 
$\mathfrak{m}'=\mathfrak{m} \cap \mathbb{T}' _{\psi ,\mathcal{O}} (U)$. 
Then there exists a continuous representation 
$\rho _{\mathfrak{m}'} :G_{F,S} \to 
GL _2 (\mathbb{T}' _{\psi ,\mathcal{O}} (U)_{\mathfrak{m}'})$ 
such that the characteristic polynomial of 
$\rho _{\mathfrak{m}'} (\mathrm{Frob}_v)$ is 
$X^2 -T_v X + \mathbf{N} (v) S_v$ for $v \notin S$. 
Here $\mathbf{N} (v)$ denotes the order of 
the residue field at $v$.
Let $\F$ be the residue field of 
$\mathbb{T}' _{\psi ,\mathcal{O}} (U)_{\mathfrak{m}'}$.
Let 
$\bar{\rho} _{\mathfrak{m}'} :G_{F,S} \to GL _2 (\F)$ 
denote the representation obtained by reducing 
$\rho _{\mathfrak{m}'}$ modulo $\mathfrak{m}'$. 

Now we suppose that $\bar{\rho} _{\mathfrak{m}'}$ 
satisfies the following conditions.
\begin{enumerate}
 \item 
 $\bar{\rho} _{\mathfrak{m}'}$ is unramified 
 outside the primes of $F$ dividing $p$. 
 \item
 The restriction of $\bar{\rho} _{\mathfrak{m}'}$ 
 to $G_{F(\zeta _p)}$ is absolutely irreducible.
 \item
 If $p=5$, and $\bar{\rho} _{\mathfrak{m}'}$ has 
 projective image isomorphic to 
 $PGL_2 (\F _5)$, 
 then the kernel of 
 $\proj \bar{\rho} _{\mathfrak{m}'}$ 
 does not fix $F(\zeta _5)$. 
 \item
 For each finite prime $v\in S\setminus \Sigma _p$, we have 
 \[
 \bigl(1-\mathbf{N}(v)\bigr)
 \Bigl(
 \bigl(1+\mathbf{N}(v)\bigr)^2 
 \det \bar{\rho} _{\mathfrak{m}'}(\mathrm{Frob}_v )-
 \bigl(\mathbf{N}(v)\bigr)
 \bigl(\tr \bar{\rho}_{\mathfrak{m}'}(\mathrm{Frob}_v)\bigr)^2 
 \Bigr) \in \F^{\times}.
 \]
\end{enumerate}
Let $R_{F,S}$ (resp. $R_{F,S} ^{\square}$) be 
the universal deformation $\mathcal{O}$-algebra 
(resp. the universal framed deformation $\mathcal{O}$-algebra) 
of $\bar{\rho} _{\mathfrak{m}'}$, 
and put 
$\mathbb{T}^{\square} =R_{F,S} ^{\square} \otimes _{R_{F,S}} 
\mathbb{T}_{\psi ,\mathcal{O}} (U)_{\mathfrak{m}}$. 
We take a subset $\sigma '$ of the set of 
primes of $F$ dividing $p$, 
and an unramified character $\chi _{\mathfrak{p}}$ 
of $G_{F_{\mathfrak{p}}}$ for each $\mathfrak{p} \in \sigma '$, 
such that $\mathfrak{m}$ is $\sigma$-ordinary 
when we put 
$\sigma =(\sigma ',\{\chi _{\mathfrak{p}} \}_{\mathfrak{p}\in \sigma '} )$.
Now we can define a deformation ring 
$\tilde{R} _{F,S} ^{\sigma ,\psi}$ and a map
$\tilde{R} _{F,S} ^{\sigma ,\psi} \to \mathbb{T}^{\square}$ 
as in (3.4) of \cite{Kis}.

\begin{thm}
With the above notation and the assumptions, 
$\tilde{R} _{F,S} ^{\sigma ,\psi} \to \mathbb{T}^{\square}$ 
is an isomorphism up to $p$-power torsion kernel.  
\end{thm}

\begin{proof}
 Applying the Theorem \ref{main}, 
the proof goes on 
as in the proof of \cite[Theorem 3.4.11]{Kis}. 
\end{proof}


\begin{thebibliography}{Ray}

\bibitem[Gee]{Gee}
T.~Gee, 
\emph{A modularity lifting theorem 
 for weight two Hilbert modular forms}, 
Math.\ Res.\ Lett.\ \textbf{13} (2006), no.~5, 805--811.

\bibitem[Kis]{Kis}
M.~Kisin, 
\emph{Moduli of finite flat group schemes, and modularity}, 
to appear in Ann.\ of Math.

\bibitem[Ray]{Ray}
M.~Raynaud, 
\emph{Sch\'{e}mas en groupes de type $(p,\ldots,p)$}, 
Bull.\ Soc.\ Math.\ France \textbf{102} (1974), 241--280.

\end{thebibliography}
\end{document}